%% file: NEWdd.tex
\documentclass{amsart}

\usepackage{amsmath,amscd,amssymb}
\usepackage[mathscr]{eucal}
\usepackage[all]{xy}

\input{hatomakro}

\numberwithin{equation}{section}

\newtheorem{theorem}{Theorem}[section]
\newtheorem{proposition}[theorem]{Proposition}
\newtheorem{corollary}[theorem]{Corollary}
\newtheorem{lemma}[theorem]{Lemma}

\theoremstyle{definition}

\newtheorem{example}[theorem]{Example}

\theoremstyle{remark}
\newtheorem{remark}[theorem]{Remark}

\numberwithin{equation}{section}
\numberwithin{theorem}{section}

\newcommand{\rmo}{\mathrm{o}}
\newcommand{\rmO}{\mathrm{O}}

\begin{document}

\title[What does a rate in a mean
ergodic theorem imply ?]{What does a rate in a mean ergodic theorem
imply?}

\author{Alexander Gomilko}
\address{Faculty of Mathematics and Computer Science\\
Nicholas Copernicus University\\
ul. Chopina 12/18\\
87-100 Toru\'n, Poland
}

\email{gomilko@mat.umk.pl}

\author{Yuri Tomilov}
\address{Faculty of Mathematics and Computer Science\\
Nicholas Copernicus University\\
ul. Chopina 12/18\\
87-100 Toru\'n, Poland\\
and
Institute of Mathematics\\
Polish Academy of Sciences\\
\'Sniadeckich 8\\
00-956 Warsaw, Poland
}

\email{tomilov@mat.umk.pl}

\thanks{  The authors were  partially
supported by the NCN grant DEC-2011/03/B/ST1/00407.}



\date{\today}

\begin{abstract}
We develop a general framework for the inverse mean ergodic
theorems with rates for operator semigroups thus completing a
construction of the theory initiated in \cite{GHT1} and
\cite{GHT}.
\end{abstract}

\maketitle

\section{Introduction}
In this paper we are concerned with  the rates of convergence of
Ces\'aro means
\begin{equation}\label{mean} \Ce_t(A):=\frac{1}{t} \int_0^t T(s)\, d{s}, \quad t >0,
\end{equation}
for a bounded $C_0$-semigroup $(T(t))_{t \ge 0}$ with
generator $-A$  on a
(complex) Banach space $X.$
Recall that
 in general
\begin{align}\label{cc}
 \{x\in X: \Ce_t(A)x \, \, \text{strongly converges}\}=\ker (A)\oplus\cls{\ran} (A).
 \end{align}
Moreover, $ \Ce_t(A)x$  converges to zero  if and only if
$x \in \cls{\ran} (A)$.

A mean ergodic theorem provides conditions under which the means in \eqref{mean} converge strongly on the whole of $X$.
One of the most well-known mean ergodic theorems says that if $X$ is reflexive then
\begin{equation*}
X=\ker (A)\oplus\cls{\ran} (A),
\end{equation*}
 hence $\Ce_t(A), t >0, $  are strongly convergent.
Mean ergodic theorems is a classical chapter of the ergodic theory and for its basic results one may consult \cite{Kr85}.

If a mean ergodic theorem holds then it is natural to try to equip it with a certain convergence rate.
After a simple normalization, one can assume without loss of
generality, that $\Ce_t(A), t >0,$ converge to zero as $t \to \infty.$ Thus we will study the  decay rates of $\| \Ce_t(A)x\|, x \in X.$
 (See the introduction in  \cite{GHT} for a more detailed discussion.)

The rates in ergodic theorems were studied in many settings and backgrounds. For some
of the achievements in this area one may consult the survey papers \cite{Ka96}, \cite{BuGe95} (and references therein)
and also  \cite{AsLi07}, \cite{CoLi03}--\cite{Der06}, \cite{KaRe10}, \cite{Shaw} and \cite{West1998}. However no systematic approach to {\it characterizing}  rates in  {\it mean}
ergodic theorems was proposed until very recent time. The present paper provides one more step towards such a characterization.
 It is a companion to our previous articles \cite{GHT} and
\cite{GHT1} where the theory of rates in mean ergodic theorems was
developed by methods of functional calculus. It was our
initial idea that a functional calculus approach  might produce certain
rates of decay of Ces\'aro means in a canonical way and thus
would allow us to quantify their convergence properties.
This idea appeared to be fruitful and opened a door to many
tools from outside  of ergodic theory. Taking advantage of these
new tools we are now able to introduce and study in details
abstract inverse theorems on decay rates, the main subject of this
paper.

To set the scene, let us first recall certain direct theorems on rates
obtained in \cite{GHT1} and \cite{GHT}. The direct problem in the
study of rates for Ces\' aro means can be formulated as follows.

{\it Direct Problem:} Given $x$ from the range (or the domain) of
a function of $A$  find a rate of decay (if any) for $\Ce_t(A)x$
and prove its optimality. (Of course, we should specify what we
mean by `function' and `optimality' and that will be clear from
further considerations.)

Theorems answering the direct problem will be called {\it direct
mean ergodic theorems with rates.} Motivated by probabilistic applications, the problem of obtaining various
direct theorems with rates in has attracted considerable attention last years.
We note the foundational paper  \cite{DerLin01} and then the subsequent papers \cite{AsLi07}, \cite{CoLi03}--\cite{Der06}, \cite{KaRe10}.

Recently, we proposed in \cite{GHT1} and \cite{GHT} an abstract
framework which allowed us to encompass many partial results and
to solve certain open problems on the rates of decay of Ces\' aro means.
In particular, we proved in \cite[Theorem 3.4 and Proposition
4.2]{GHT}  that if $(T(t))_{t \ge 0}$ is a bounded $C_0$-semigroup on $X$ and $f$ is a  Bernstein function, $\lim_{t \to
0+}f(t)=0,$ then
\begin{equation}
x \in \ran (f(A))\Longrightarrow \norm{\Ce_t(A)x} =\mbox{O}(f(t^{-1})),\qquad t\to\infty.
\label{ostill}
\end{equation}
As corollaries, we obtained  rates of decay of Ces\'aro means
on the ranges of polynomial and logarithmic functions thus extending and
sharpening known results. Our results were proved to be optimal in
a natural sense.

To understand the limitations of direct mean ergodic theorems with
rates it is natural to ask whether the implication
 \eqref{ostill}  can be reversed. Examples
show that
one cannot in general expect the implication opposite to \eqref{ostill}  to be true (see e.g. Section \ref{optimality}
of the present paper  and   \cite[Example, p. 121]{DerLin01} concerning the discrete setting). Thus we are interested in the best possible conditions
on the decay of the means implying the converses of  \eqref{ostill}, and our abstract inverse problem reads as follows.

{\it Inverse Problem:} Given the rate of decay of $\Ce_t(A)x$ for an element $x \in X$ prove that $x$  is in
an appropriate range (or domain) of a function of $A$  and show
optimality of the result.

Statements of that form will be called {\it inverse mean ergodic theorems with
rates}. The first inverse theorems were proved in the discrete setting  by  Browder \cite{Bro58} and
Butzer and Westphal \cite{BuWe71}. They showed (indirectly in the first case) that if $X$ is reflexive and $T$ is a power bounded operator on $X$, then
$\norm{1/n \sum_{k=0}^{n-1} T^k x}=\rmO({1}/{n})$ implies that $x \in  \ran (I-T).$
It was also noted in \cite{BuWe71} that one cannot produce better
rates than $1/n,$ since $\norm{1/n \sum_{k=0}^{n-1} T^k x}=\rmo({1}/{n})$ implies $x=0.$ Thus one has to deal with rates between $1/n$
and $\rmo(1),$ and the same is true for the continuous time means  $\Ce_t(A)x$ when the rate $1/n$ is replaced by $1/t$ - see
\cite{GHT} for a discussion.
This complicates the study of rates since many plausible conditions involving rates appear to be too strong in view of
the extremal $1/n$ (or $1/t$) property. See e.g. our Appendix.

Various partial situations (mostly of polynomial rates and mostly in the discrete framework) were
considered in \cite{Cu10}, \cite{CuLi09}, \cite{CoCuLi11}, \cite{DerLin01}, and \cite{KaRe10}. The main goal of the present paper is to provide
an abstract set-up for the inverse theorems and to give them a
systematic treatment. This set-up appears to be coherent with
direct theorems obtained in  \cite{GHT} and it constitutes
in a sense a final block of the theory developed in
\cite{GHT}. As in the case of direct theorems treated in
\cite{GHT}, known inverse theorems on rates for
particular cases (e.g. for polynomial rates) can be included in
our framework.

In \cite{GHT} our direct ergodic theorems involved the ranges of complete Bernstein functions of semigroup generators
 (as e.g. in \eqref{ostill}).
In the present study of the inverse theorems, it will  be convenient to
restrict our attention to the class of Stieltjes functions and to
deal with their domains rather than the ranges of (reciprocal) operator complete Bernstein functions.
Such a setting enabled us to apply
an (adapted) abstract characterization of the domains of operator complete
Bernstein functions due to  Hirsch \cite{HirschFA}. (A similar result
in a slightly more general setting was obtained later by R.
Schilling, see e.g. Theorem $12.19,$  Remark $12.20,$ and Corollary
12.21 in \cite{SchilSonVon2010} and also \cite{Schil98}.)

The paper is based on ideas worked out in \cite{GHT1}, \cite{GHT},
and \cite{HaTo10}. However, its finer details are essentially
different from the arguments used in those papers and it
complements the results obtained in \cite{GHT1}, \cite{GHT}, and
\cite{HaTo10}. To give a flavor of inverse theorems proved by our
technique we indicate a partial converse of the direct theorem
formulated above. It illustrates our approach of adding an `extra
rate' to the decay of the means in order to invert the direct
statements.

Assume that $\cls{\ran} (A)=X.$ If $g$ is a Stieltjes function of
the form
\begin{equation*}
g(z)=\int_{0+}^\infty \frac{\mu(ds)}{z+s}, \qquad z>0,
\end{equation*}
where $\mu$ is a (non-negative) Radon measure on $(0,\infty)$ such
that
\begin{equation*}
 \int_{0+}^\infty \frac{\mu(ds)}{1+s}\,<\,\infty,
\end{equation*}
$g(0+)=\infty,$ and  $x\in X$ satisfies
\begin{equation*} \int_1^\infty
\frac{g(1/t)\|\Ce_t(A)x\|}{t} dt<\infty,
\end{equation*}
then $x\in \dom(g(A))$, or, equivalently, $x \in \ran ([1/g] (A)).$

Note that there are close relations between inverse mean ergodic
theorems for bounded $C_0$-semigroups $(T(t))_{t \ge0}$ and
bounded discrete semigroup $(T^n)_{n \ge 0}$, and our approach, in
fact, unifies continuous and discrete frameworks. It allows one to
study the continuous and the discrete cases simultaneously and to
obtain results parallel in spirit and proofs. However, because of
space limitations, the functional calculus approach to inverse
mean ergodic theorems in the discrete case will be presented
elsewhere.

 We also show
that our statements are sharp and cannot
in general be improved. In fact, it appears that they are
are optimal even for a very simple multiplication operator on an
$L_1$ space. However, even in this simple case, there are
nontrivial technical difficulties to overcome. Thus a substantial
part of the paper  is devoted to proving optimality of our results
in various senses.

Our Appendix  addresses important and  related to inverse
theorems problem which however stay a bit aside from the
mainstream of the exposition and thus shifted to a separate part.
We believe that it is of independent interest.
There we prove that the means cannot be too small in an ``integral'' sense.

\subsection{Some Notations and Definitions}\label{notations}
For a closed linear operator $A$ on a complex Banach space $X$ we
denote by $\dom(A),$ $\ran(A)$, $\ker(A)$, and $\rho(A)$  the
{\em domain}, the {\em range},  the {\em kernel}, and the {\em resolvent set} of $A$, respectively. The norm-closure of the range is
written as $\cls{\ran}(A)$. The space of bounded linear operators
on $X$ is denoted by $\Lin(X)$.
Finally, we set $\mathbb R_+:=[0,\infty)$ and $\mathbb C_+:=\{\lambda \in \mathbb C: {\rm Re \, \lambda} > 0 \}.$

\section{Preliminaries}\label{prem}
\subsection{Functional calculus: Bernstein and Stieltjes functions}
In this subsection we recall basic properties of operator
Bernstein and Stieltjes functions and prove several auxiliary
statements on functional calculi useful for the sequel. Moreover
we arrange the material in the way most suitable for our purposes.
The developed machinery will be used intensively in the next sections.

Let $\eM(\R_+)$ be a Banach algebra of bounded Radon measures on $\R_+.$
Define the {\em Laplace transform} of  $\mu \in \eM(\R_+)$  as
\[ (\Lap\mu)(z) := \int_{\R_+} e^{-sz} \, \mu(\ud{s}),
\qquad  z \in \mathbb C_+,
\]
and note that $\Lap\mu$ extends to a  continuous function on $\cls{\C}_+$.
Note that the space
\[ \Wip(\C_+) := \{ \Lap\mu : \mu \in \eM(\R_+)\}
\]
is a commutative Banach algebra with pointwise multiplication and with respect to the
norm
\begin{equation}\label{mmm}
\norm{\Lap \mu}_{\Wip} := \norm{\mu}_{\eM(\R_+)} = \abs{\mu}(\R_+),
\end{equation}
and the Laplace transform
\[ \Lap : \eM(\R_+) \pfeil \Wip(\C_+)
\]
is an isometric isomorphism.

Let $-A$ be the generator of a bounded $C_0$-semigroup
$(T(s))_{s\ge 0}$ on a Banach space $X$. Then the mapping
\[ g = \Lap {\mu} = \int_{\R_+} e^{-s\cdot}\, \mu(\ud{s}) \quad
\mapsto \quad g(A) := \int_{\R_+} T(s)\, \mu(\ud{s})
\]
(where the integral converges in the strong topology)
is a continuous algebra homomorphism of $\Wip(\C_+)$ into $\Lin(X).$
The homomorphism is called the {\em Hille-Phillips (HP-) functional calculus}
for $A$. Its basic properties can be found in  \cite[Chapter XV]{HilPhi}.

The HP-calculus has an extension to a larger function
class. This extension is
constructed  as
follows: if $f: \C_+ \to \C$ is holomorphic such that there exists
a function $e\in \Wip(\C_+)$ with $ef \in \Wip(\C_+)$ and the
operator $e(A)$ is injective, then we define
\[ f(A) := e(A)^{-1} \, (ef)(A)
\]
with its natural domain $\dom (f(A)):=\{x \in X :
(ef)(A)x \in \ran(e(A)) \}$. In this case $f$ is called
{\em regularizable}, and $e$ is called a {\em regularizer} for $f$. Such a
 definition of $f(A)$ does not depend on the
particular regularizer $e$ and $f(A)$ is a closed
 operator on $X$. Moreover, the set of all
regularizable functions $f$ is an algebra depending on $A$
(see e.g. \cite[p. 4-5]{Haa2006} and \cite[p. 246-249]{deLau95}),
and the mapping
\[ f \tpfeil f(A)
\]
from this algebra into the set of all closed operators on $X$ is
called the {\em extended Hille--Phillips calculus} for $A$.
The  next {\em
product rule} of this calculus (see e.g. \cite[Chapter 1]{Haa2006}) will be crucial for the sequel: {\em if $f$ is regularizable and $g\in \Wip(\C_+)$,
then}
\begin{equation}\label{hpfc.e.prod}
 g(A) f(A) \subseteq f(A) g(A) = (fg)(A),
\end{equation}
where we take the natural domain for a product
of operators.

This regularization approach can be  applied to the study of
operator Bernstein functions. First recall that  $f\in \Ce^\infty(0, \infty)$ is called a complete monotone function  if
\[
f(t)\geq 0\quad \mbox{and}\quad (-1)^n \frac{d^n
f(t)}{d{t}^n}\ge 0 \qquad \text{for all $n \in \mathbb N$ and $
t > 0$}.
\]
A function $f\in \Ce^\infty(0, \infty)$ is called {\em Bernstein
function} if its derivative is completely monotone.
By \cite[Theorem 3.2]{SchilSonVon2010}, $f$ is
Bernstein  if and only if there exist  constants $a,
b\geq 0$ and a positive Radon measure $\mu$ on $(0,\infty)$
satisfying
\begin{equation*}
\int_{0+}^\infty\frac{s}{1+s}\,\mu(d{s})<\infty \label{mu}
\end{equation*}
and such that
\begin{equation}\label{hpfc.e.bf}
f(z)=a+bz+\int_{0+}^\infty (1-e^{-sz})\mu(d{s}), \qquad z>0.
\end{equation}
The formula \eqref{hpfc.e.bf} is called the Levy-Khintchine
representation of $f.$ The triple $(a, b, \mu)$ is uniquely
determined by the corresponding Bernstein function $f$ and is
called the Levi-Khintchine triple.

It was proved in \cite[Lemma 2.5]{GHT} that  Bernstein functions belong to
the extended HP-functional calculus and every Bernstein function is
regularizable by any of the functions $e_{\lambda}(z)=(\lambda
+z)^{-1}$, $\re \lambda
> 0.$
\vanish{Namely
the
following statement holds.
\begin{lemma}\label{hpfc.l.bf-fc}
Every Bernstein function $f$ can be written in the form
\[ f(z) = f_1(z) + z\, f_2(z), \qquad z >0,
\]
where $f_1, f_2 \in \Wip(\C_+)$.
\end{lemma}
Lemma \ref{hpfc.l.bf-fc} implies that
every Bernstein function is
regularizable by any of the functions $e_{\lambda}(z)=(\lambda
+z)^{-1}$, $\re \lambda
> 0.$
} This led to the following the operator Levy-Khintchine
representation for a Bernstein function $f$ of $A$ (cf.
\eqref{hpfc.e.bf}) essentially due to Phillips \cite{Phil52}.
\begin{theorem}\label{hpfc.c.bf-fc}
Let $-A$ generate a bounded $C_0$-semigroup $(T(s))_{s\ge 0}$ on $X$, and let $f \sim (a,b,\mu)$ be a Bernstein
function. Then $f(A)$ is defined in the extended HP-calculus. Moreover, $\dom(A) \subseteq \dom(f(A))$ and
\begin{equation}\label{phillips}
f(A)x = ax + bAx + \int_{0+}^\infty (I - T(s))x \, \mu(d{s})
\end{equation}
for each $x\in \dom(A),$ and $\dom(A)$ is a core for $f(A)$. If
 $a
> 0,$ then $\ran(f(A)) =X$ and $f(A)$ is invertible.
\end{theorem}
For the detailed theory of operator Bernstein functions we refer to \cite{SchilSonVon2010}.

The class of Bernstein functions is quite large and to ensure good
algebraic and function-theoretic properties of  Bersntein
functions it is convenient and also sufficient for  many purposes
to consider its subclass consisting of complete Bernstein
functions.  A Bernstein function is called  {\em complete} if its
representing measure in the Levy-Khintchine formula
\eqref{hpfc.e.bf}  has a completely monotone density with respect
to Lebesgue measure, see \cite[Definition 6.1]{SchilSonVon2010}.

To discuss other representations of complete Bernstein functions,
more suitable for the goals of this paper, we will also need yet
another related class of  functions.
 A function $g : (0, \infty) \to \R_+$
is called  {\em Stieltjes} if it can be written as
\begin{equation}\label{hpfc.e.stieltjes}
g(z) = a + \frac{b}{z} + \int_{0+}^\infty \frac{\mu(d{s})}{z +
s}, \qquad z > 0,
\end{equation}
where $a, b \ge 0$ and $\mu$ is a positive Radon measure on $(0,
\infty)$ satisfying
\begin{equation}\label{mmu}
 \int_{0+}^{\infty}
\frac{\mu(d{s})}{1 + s} < \infty.
\end{equation}

 Since the
representation \eqref{hpfc.e.stieltjes} is unique, the measure
$\mu$ is called a {\em Stieltjes measure} for $g$ and
\eqref{hpfc.e.stieltjes} is called the {\em Stieltjes
representation} for $g$, see e.g. \cite[Chapter
2]{SchilSonVon2010}. We will then write $g \sim (a,b, \mu).$
Note that
\begin{equation}\label{ab}
a=g(\infty), \qquad b = \lim_{z \to 0+} z g(z).
\end{equation}

The following result (see \cite[Theorem 6.2 and Corollary
7.4]{SchilSonVon2010}) shows, in particular, a reciprocal duality between
complete Bernstein and Stieltjes functions, and it will be crucial for the sequel.
\begin{theorem}\label{hpfc.t.sti-char}
A non-zero function $g$ is a Stieltjes function if and only $z
g(z), z>0,$ is a complete Bernstein function, if and only if $1/g$
 is a complete Bernstein function.
 \end{theorem}
\begin{remark}
Thus every complete Bernstein function $f$ admits a unique representation
\begin{equation}\label{compbern}
f(z) = {a} + b z + \int_{0+}^\infty \frac{z }{z + s}\, \mu(d{s}),
\qquad z > 0,
\end{equation}
where $a, b \ge 0$ and $\mu$ is a positive Radon measure on $(0,
\infty)$ satisfying \eqref{mmu},
and we can speak of the {\it Stieltjes representation} $(a,b,\mu)$
of  $f,$ and  write $f \sim (a,b,\mu).$

However, there are also other representations for complete
Bernstein functions in the literature. For example, we note
the representation
\begin{equation}\label{compbern1}
f(z) = {a} + b z + \int_{0+}^\infty \frac{z \nu(d{s})}{1 +zs},
\qquad \int_{0+}^{\infty}\frac{\nu(d{s})}{1 +s} <\infty,
\end{equation}
used in particular in \cite{HirschInt} and \cite{HirschFA}. The
representations \eqref{compbern} are \eqref{compbern1} are
equivalent in the sense that one of them is transformed by the
change of variable $s=1/t$ into another so that the measures $\mu$
and $\nu$ satisfy the same integrability condition \eqref{mmu}.
\end{remark}

We will be interested in Stieltjes functions $g$  with the Stieltjes representation of the form $(0,0,\mu),$ and satisfying $g(0+)=\infty.$
Before going further, we give several elementary examples of such functions  important for the sequel.

\begin{example}\label{contex}
a) The functions $g_{\gamma}(z):=z^{-\gamma}, \gamma \in (0,1),$ are Stieltjes and
\[
\lim_{s\to 0+}\,g(s)=\infty, \quad g_{\gamma}(z)=\frac{\sin \pi\gamma}{\pi} \int_0^\infty
\frac{ds}{(z+s)s^\gamma}, \quad z >0.
\]
Accordingly, $f_{\gamma}(z)=zg_{\gamma}(z)=z^{1-\gamma}, \gamma \in (0,1), $ are complete Bernstein functions.

b) By  \cite[Example 2.9]{GHT} the function
\[
g(z):=\frac{\log z}{z-1}=
\int_0^\infty\frac{ds}{(z+s)(1+s)}, \qquad z >0,
\]
is  Stieltjes with $g(0+)=\infty,$ and so is the  function
$
g(z):=\frac{z-1}{z\log z}
$
by Theorem \ref{hpfc.t.sti-char}.
\end{example}

Let us show now that complete Bernstein and Stieltjes functions of the
generator $-A$
can be expressed in resolvent terms in accordance
with the formulas \eqref{hpfc.e.stieltjes} and \eqref{compbern}.
To this aim, we will need the notion of a sectorial operator. Recall that a linear operator $V$ on $X$ is called {\em sectorial}
if $(-\infty,0)\subset \rho(V)$ and there exists $c>0$ such that
\[
s \|(s+V)^{-1}\|\leq {c},\qquad s>0.
\]
\begin{theorem}\label{resolventbern}
Let  $-A$ be the generator of a bounded $C_0$-semigroup on $X$.
\begin{itemize}
\item [{\it (i)}] If $f\sim (0,0,\mu)$   is a complete Bernstein function, then
\begin{equation}\label{completeb}
f(A)x=\int_{0+}^{\infty}{A}(s+A)^{-1}x \,\mu(ds)
\end{equation}
for every $x \in \dom (A).$ Moreover, $\dom (A)$ is a core for $f(A).$
\item [{\it (ii)}] If $g\sim (0,0,\mu)$ is  a Stieltjes function and $A$ has dense range, then
$g$ belongs to the extended HP-calculus and
\begin{equation}\label{stiltcalc}
g(A)x=\int_{0+}^{\infty}(s+A)^{-1}x \,\mu(ds)
\end{equation}
for every $x \in \ran (A).$ Moreover, $\ran(A)$ is a core for $g(A).$
\end{itemize}
\end{theorem}
\begin{proof}
The proof of \eqref{completeb} relies on a direct transforming \eqref{phillips} to the form \eqref{completeb}
 by means of the definition of a complete Bernstein function
and can be found in \cite[p. 149]{SchilSonVon2010}. The fact that $\dom(A)$ is a core for $f(A)$ follows from
Theorem \ref{hpfc.c.bf-fc}. Thus {\it (i)} is a straightforward consequence of Theorem \ref{hpfc.c.bf-fc}.

To prove {\it (ii)} note that by Theorem \ref{hpfc.t.sti-char} if
$g$ is a Stieltjes function then $g(z)=q(z)/z$ for some complete
Bernstein function $q.$ Since by \cite[Lemma 2.5]{GHT} $q$ is
regularizable by $1/(z+1)$ and $A$ is injective in view of
\eqref{cc}, the function  $g$ is regularizable by ${z}/({z+1})$
and belongs to the extended HP-calculus. Moreover, if $x \in \ran
(A),$ then using the product rule for the extended HP-calculus we
obtain
\begin{eqnarray*}
g(A)x&=& \left[\frac{z+1}{z} \cdot z g \cdot  \frac{1}{z+1} \right](A)x\\
&=& (A+I)A^{-1} \int_{0+}^{\infty}{A}(s+A)^{-1} (A+I)^{-1} x \,\mu(ds)\\
&=& \int_{0+}^{\infty}(s+A)^{-1} x \,\mu(ds).
\end{eqnarray*}

It remains to prove that $\ran (A)$ is a core for $g(A).$ Observe that by e.g. \cite[Proposition 2.2.1,b]{Haa2006} the operator $A^{-1}$ is sectorial with dense domain
$\ran(A).$ Hence
if
$e_t(A)=t(t+ A^{-1})^{-1},$ $t>0,$ then $e_t(A) x \to x$ for every
$x \in X$ as  $t \to \infty$ by \cite[Proposition 2.2.1, c]{Haa2006}.
Since $e_t \in \Wip(\C_+)$  for each $ t>0$, the product rule \eqref{hpfc.e.prod}
implies that
if $x \in \dom
(g(A))$ and $g(A)x = y$ then $g(A)e_t(A)x = e_t(A)y.$ As $\ran
(e_t(A))= \dom (A^{-1})=\ran(A),$ the statement follows.
\end{proof}

Let $g$ be a Stieltjes function with the Stieltjes representation
$(0,0,\mu),$ i.e.
\begin{equation}\label{fung}
g(z)=\int_{0+}^\infty \frac{\mu(ds)}{z+s}, \quad z>0,\quad
\int_{0+}^{\infty} \frac{\mu(ds)}{1+s}<\infty.
\end{equation}
If a complete Bernstein function $h$ is given by
\begin{equation}\label{berhirsch}
h(z):=g(1/z)=\int_{0+}^\infty \frac{z\,\mu(ds)}{1+zs}.
\end{equation}
 and a linear operator $V$ is  sectorial
then the operator $h(V)$ can be defined  as  the
closure of a (closable) linear operator $h_0(V)$ given  by the
formula
\begin{equation}\label{hirschform}
h_0(V)x=\int_{0+}^\infty V(1+sV)^{-1}x\,\mu(ds),\quad x\in
\dom(V).
\end{equation}
This definition is due to Hirsch and it was introduced and thoroughly studied in his paper  \cite{HirschInt}.
Thus $h(V)$ is a closed linear operator and $\dom
(h(V))\supset \dom (V).$

 If $A^{-1}$ is a sectorial operator  with  dense domain
$\ran(A),$ then setting $V=A^{-1}$ in \eqref{hirschform}
we obtain
\begin{eqnarray}
\label{relationh} h(A^{-1})x =\int_{0+}^\infty
(A+s)^{-1}x\,\mu(ds) = g(A)x,\qquad x\in \ran(A).
\end{eqnarray}
Hence the operators $h(A^{-1})$ and $g(A)$ coincide on their core
$\dom(A^{-1})=\ran(A)$ and therefore coincide. In other words, $g(A)$ defined in the extended HP-calculus
coincides with $h(A^{-1})$ defined by means  of \eqref{hirschform}.

Hirsch proved in \cite{HirschInt}  a number of properties of complete
Bernstein functions of sectorial operators.
We will need two of them which we state as a lemma. For their proofs
see \cite[Theorem 1]{HirschInt}
and \cite[Theorem 3]{HirschInt}.
\begin{lemma}\label{Hirschl}
Let $f$ and $q$ be complete Bernstein functions and let $A$ be a
sectorial operator with dense range. Then
\begin{itemize}
\item [{\it (i)}] $f(A)$ is a sectorial operator with dense range;
\item [{\it (ii)}] $(f \circ q)(A)=f(q(A)).$
\end{itemize}
\end{lemma}

The property \eqref{relationh}
will allow us to link several results from \cite{HirschInt} and \cite{HirschFA} to our
setting of Stieltjes functions of semigroup generators.
\begin{lemma}
Let  $-A$ be the generator of a bounded $C_0$-semigroup on $X,$ with dense range. If $f$ is a complete Bernstein function and
$g$ is a Stieltjes function, then their composition $f\circ g$ belongs to the extended HP-calculus, and $f(g(A))=(f\circ g) (A).$
As a consequence,
\begin{equation}\label{inclusion}
\dom ((f\circ g)(A))
\supset \dom (g(A)).
\end{equation}
\end{lemma}
\begin{proof}
By  assumption, $A^{-1}$ is sectorial with dense domain. Note that if $g$ is a Stieltjes function and a complete Bernstein function  $h$ is given by
\eqref{berhirsch}, then  $g(A)=h(A^{-1}),$ where $h(A^{-1})$ is defined by means of \eqref{hirschform}. By Lemma \ref{Hirschl},{\it (i)}
the operator $h(A^{-1})$ is  sectorial with dense range, so
$g(A)$ is the same.

 Observe also that  the composition $f \circ g$ is Stieltjes, see e.g \cite[Theorem 7.5]{SchilSonVon2010}.
Then  using Lemma \ref{Hirschl}, {\it (ii)} and \eqref{relationh}
we conclude that
$$f(g(A))=f(h(A^{-1}))=(f \circ h)(A^{-1})=[(f \circ h)(1/z)](A)=(f\circ g) (A).$$
 This implies, in particular, by Theorem \ref{hpfc.c.bf-fc} that
 $\dom ((f\circ g)(A))=\dom f (g(A))
\supset \dom (g(A)).$
\end{proof}

The next statement describing  domains of operator Stieltjes
functions in resolvent terms  is basic for the paper. It is in fact
a reformulation of \cite[Theorem 2]{HirschFA} based on
\eqref{relationh}.
\begin{theorem}[Hirsch Criterion]\label{stilprop1} If $-A$ is the generator of a bounded
$C_0$-semigroup on $X$ such that $\cls{\ran}(A)=X$ and
$g$ is a Stieltjes function given by \eqref{fung}, then
\begin{eqnarray*}
x\in \dom(g(A))\,&\iff&\,\,
\exists\,\,
\text{weak}-\lim_{\delta \to 0+}\int_\delta^1 (A+s)^{-1}x\,\mu(ds)
\\
&\iff&\,\,
\exists\,\,
\text{strong}-\lim_{\delta \to 0+}\int_\delta^1 (A+s)^{-1}x\,\mu(ds).
\label{stilvec}
\end{eqnarray*}
\end{theorem}

\section{Rates in mean ergodic theorem}\label{section3}

For the whole of this section, we make the following assumptions:
\begin{align*}
 &-A \,\, \text{\emph{is the generator of a}}\,\, C_0-\text{\emph{semigroup}} (T(t))_{t \ge 0},\\
 &M:=\sup_{t \ge 0}\norm{T(t)}<\infty, \,\, \text{\emph{and}} \,\, \, \overline{\ran} (A)=X,
 \end{align*}
 (recall that in this case,  by \eqref{cc}, ${\rm ker} (A)=\{0\}$) and
 \begin{align*}
g \,\,\text{\emph{is a Stieltjes function}}, \,\,\, g \sim (0,0,\mu),\quad  g(0+):=\infty.
\end{align*}

Let us comment on the above assumptions on $g$. For technical reasons, it
will be  more convenient for us to  consider  Stieltjes functions
as above than  those  of the general form \eqref{hpfc.e.stieltjes}, \eqref{mmu}.
To see that we do not loose generality indeed, note that
we may assume  $a=0$ (that is  $ \lim_{s\to\infty}\,g(s)=0)$
in view of  $\dom(g(A)+a)=\dom (g(A))$.
If $b\not =0$ then by passing to the reciprocal complete Bernstein function $1/g$ and using
\cite[Corollary 12.7]{SchilSonVon2010} we infer that
$\dom(g(A))=\ran(A),$ so that the Ces\'aro means $\Ce_t(A)x$ for $x \in \dom(g(A))$ decay at
the extremal rate:
\begin{equation}\label{extreme}
\|\Ce_t(A)x\|={\rm O} \left({t}^{-1}\right),\qquad t \to \infty,\quad x\in \dom (g(A)).
\end{equation}
The inverse theorems given below become void in this case, see Remark \ref{ApB} in Appendix.
  Finally if  $g(0+)<\infty$ then our direct theorem on rates from \cite{GHT} (see also Theorem \ref{ratebern} below) does not yield any rate of decay of $\Ce_t(A)$ restricted to
 $\dom(g(A))$ since $1/g(1/t) \not \to 0 , t \to \infty$ in this
 case.

We start with an elementary inequality which  will nevertheless
be essential in the proof of direct theorems for rates  in both
discrete and continuous cases.
\begin{lemma}\label{exp}
Let $f \sim (0,0,\mu)$  be a complete Bernstein function.
Then
\begin{equation}
\frac{1}{t}\int_{0+}^\infty \frac{1-e^{-st}}{s}\mu(ds)\leq 2f(t^{-1}),\qquad t>0.
\label{ft1}
\end{equation}
\end{lemma}
\begin{proof}
Since
\[
\frac{1-e^{-x}}{x}\leq \frac{2}{1+x},\qquad x>0,
\]
we have
\[
\frac{1-e^{-st}}{ts}\,\leq\frac{2}{1+ts}=\frac{2t^{-1}}{t^{-1}+s},\qquad
s,t>0.
\]
Hence
\[
\frac{1}{t}\int_{0+}^\infty \frac{1-e^{-st}}{s}\mu(ds)\leq
2\int_{0+}^\infty \frac{t^{-1}\,\mu(ds)}{t^{-1}+s}=2f(t^{-1}),\qquad t>0.
\]
\end{proof}

First we derive a convergence rate for $\Ce_t(A)x$ for $x \in
\dom(g(A)),$ or equivalently for $x \in \ran ([1/g](A))$ with $1/g$
being a complete Bernstein function.  The following theorem is a
partial case of \cite[Theorem 3.4 and Proposition 4.2]{GHT} where
the convergence rate was obtained in terms of the limit behavior
of $1/g$ at zero for the whole class of Bernstein functions. However, in
the particular situation of complete Bernstein functions
we give an  argument which is simpler  and more transparent than that from \cite{GHT}. Moreover, it illustrates
nicely our functional calculus approach  and makes the presentation
self-contained.

\begin{theorem}\label{ratebern}
 If $x \in
\dom (g(A))$ then
\begin{equation}\label{rateber}
\|\Ce_t(A)x \|\leq
4M \frac{ \| g(A)x\|}{g (t^{-1})},\qquad t>0.
\end{equation}
\end{theorem}
\begin{proof}
Remark first that if $g$ a Stieltjes function, then $f=1/g$ is a
complete Bernstein function and by \cite[Theorem 1.2.2, d)]{Haa2006} one has
\begin{equation}\label{inversef}
(f(A))^{-1}=(1/f)(A)=g(A),
\end{equation}
 hence $\dom(g(A))=\ran (f(A)).$

 Let  $y \in \dom(A)\subset \dom (f(A))$ and $t>0.$ Then from \eqref{completeb}  it follows that
\begin{eqnarray*}
t\Ce_t(A) f(A)y
&=&\int_0^\infty \int_0^t T(\tau) A(A+s)^{-1}y\, d\tau\,\mu(ds)
\\
&=&\int_0^\infty [1-T(t)] (A+s)^{-1}y \,\mu(ds).
\end{eqnarray*}
Since
\begin{eqnarray*}
[I-T(t)] (A+s)^{-1}y
&=&\int_t^\infty \left(e^{-s(\tau-t)}
-e^{-s \tau}\right) T(\tau)y\,d\tau+
\int_0^t e^{-s\tau}T(\tau)y\,d\tau,
\end{eqnarray*}
we infer that
\begin{eqnarray*}
\|[1-T(t)) (A+s)^{-1}y\|
\label{coreq2}
&\leq& M\|y\|\left\{\int_t^\infty \left(e^{-s(\tau-t)}
-e^{-s \tau}\right) \,d\tau+
\int_0^t e^{-s\tau}\,d\tau\right\}\\
&=&
 2M\|y\| \cdot
\frac{1-e^{-st}}{s}.
\end{eqnarray*}

Thus using Lemma \ref{exp} we have
\begin{eqnarray*}
\|\Ce_t(A)f(A)y\|&\leq& 2M\|y\|\int_0^\infty
\frac{1-e^{-st}}{ts}\,\mu(ds)
\\
&\le&{4M} f(t^{-1})\|y\|, \qquad y \in \dom(A).
\end{eqnarray*}
Since  $\dom(A)$ is a core for $f(A),$ by passing to closures in the last inequality, we finally obtain
\begin{equation}\label{berform}
\|\Ce_t(A)f(A)y\|\le {4M}{f(t^{-1})}\|y\|, \qquad y \in \dom (f(A)).
\end{equation}
Then
\eqref{inversef} and \eqref{berform} imply
\eqref{rateber}.
\end{proof}

\begin{remark}\label{stcor1}
Note that Theorem \ref{ratebern} can be formulated in terms of
complete Bernstein functions as in \eqref{berform}.
It is this form of \eqref{rateber} that we have obtained in
\cite[Theorem 3.4 and Proposition 4.2]{GHT}.
\end{remark}

The following result, Theorem \ref{stilthh}, is our main  inverse theorem for rates in the
continuous setting. At first glance, its assumptions differ
 from the conclusions of (direct) Theorem \ref{ratebern}.
We show however that the result yields several statements which are
are ``almost'' converse of Theorem \ref{ratebern}. The word
``almost'' is crucial: we prove that the result is optimal and thus there
is an unavoidable, in general,  gap between our direct and inverse mean ergodic
theorems with rates.
\begin{theorem}\label{stilthh}
If $x\in X$ is such that
\begin{equation}
\int_1^\infty \frac{|g'(1/t)|\|\Ce_t(A)x\|}{t^2}\,dt
<\infty,\label{g1}
\end{equation}
then
$x\in \dom(g(A)).$
\end{theorem}

\begin{proof}
Note that for any $s>0$  and $x\in X$
\begin{eqnarray*}
(A+s)^{-1}x&=&\int_0^\infty e^{-s t} T(t)x\,dt= \int_0^\infty
e^{-s
t} \left(\int_0^t T(\tau)x\,d\tau\right)' \, dt\\
&=&\int_0^\infty s t e^{-st} C_t(A)x\,dt. \label{summation}
\end{eqnarray*}

Therefore
\begin{equation*}
\|(A+s)^{-1}x\|\leq M\|x\|+\int_1^\infty  s t e^{-s
t}\norm{C_t(A)x}\,dt \label{TC}
\end{equation*}
and
\begin{eqnarray*}
\int_{0+}^1 \|(A+s)^{-1}x\|\,\mu(ds)&\leq&
M\|x\|\int_{0+}^1\,\mu(ds) \\
&+&\int_1^\infty \left(\int_{0+}^1   s t e^{-s
t}\,\mu(ds)\right)\,\norm{C_t(A)x}\,dt.
\end{eqnarray*}
To estimate the inner integral observe that
for every $\tau \geq 0:$
\begin{equation}\label{exponn}
\tau e^{-\tau}\,\leq\,\frac{4}{(1+\tau)^2},
\end{equation}
since
\begin{eqnarray*}
4e^\tau -\tau(1+\tau)^2
> 4\left(\sum_{i=0}^{4} \frac{t^i}{i!}\right)
-\tau-2\tau^2-\tau^3
>4+\frac{\tau^3}{6}(\tau-2)>0.
\end{eqnarray*}
Now using \eqref{exponn} with $\tau=ts$ we have
\begin{eqnarray*}
\int_{0+}^\infty ts e^{-ts}\,\mu(ds)\,\leq\,4
\int_{0+}^\infty \frac{\mu(ds)}{(1+ts)^2}=
\frac{4}{t^2}\int_{0+}^\infty
\frac{\mu(ds)}{(1/t+s)^2}
=4\frac{|g'(1/t)|}{t^2}.
\end{eqnarray*}
Thus
\begin{eqnarray}
\label{TC0} \int_{0+}^1 \|(A+s)^{-1}x\|\,\mu(ds) &\leq&
M\|x\|\int_{0+}^1\,\mu(ds) \\
&+&4\int_1^\infty {\frac{|g'(1/t)|}{t^2}\,\norm{C_t(A)x}} \,dt <
\infty,\notag
\end{eqnarray}
and $x \in \dom(g(A))$ by  Hirsch's Theorem \ref{stilprop1}.
\end{proof}

The next direct corollary of Theorem \ref{stilthh} is formulated
in terms of a norm estimate for $\Ce_t(A)$ thus removing
assumptions on the derivative of $g.$
\begin{corollary}\label{qgt}
If $x\in X$ and a measurable function
$\epsilon :(g(1),\infty) \mapsto (0,\infty)$ satisfy
\begin{equation}\label{qucond}
\|\Ce_t(A)x\|=
\rmO\left(\frac{1}{g(1/t)\epsilon(g(1/t))}\right),\quad t\to\infty,
\qquad \int_{g(1)}^\infty \frac{d\tau}{\tau \epsilon(\tau)}\,
<\infty,
\end{equation}
then
$x\in \dom(g(A)).$
\end{corollary}

\begin{proof}
If
 \eqref{qucond}
 holds, then there exists $c>0$ such that
\[
\int_1^\infty \frac{|g'(1/t)|\|\Ce_t(A)x\|}{t^2}\,dt \le c\int_1^\infty
\frac{d g(1/t)}{g(1/t)\epsilon(g(1/t))} =c\int_{g(1)}^\infty
\frac{d\tau}{\tau \epsilon(\tau)}<\infty,
\]
and Theorem \ref{stilthh} implies $x\in \dom(g(A))$.
\end{proof}

Now we derive a corollary of Theorem \ref{stilthh} which is almost converse to  Theorem \ref{ratebern}.
It is however strictly weaker than Theorem \ref{stilthh} and at the same time it cannot
essentially be improved as we will show in Section \ref{optimality}.
\begin{corollary}\label{corfirst}
If $x\in X$ is such that
\begin{equation}
\int_1^\infty \frac{g(1/t)\|\Ce_t(A)x\|}{t} dt<\infty,
\label{first}
\end{equation}
then $x\in \dom(g(A))$.
\end{corollary}
\begin{proof}
It suffices to observe that
\begin{equation}\label{rem3}
|g{'}(\tau)|=\int_{0+}^\infty \frac{\mu(ds)}{(\tau+s)^2}\le
\frac{1}{\tau}\int_{0+}^\infty
\frac{\mu(ds)}{\tau+s}=\frac{g(\tau)}{\tau},\qquad \tau>0.
\end{equation}
(In fact, a more general estimate is given in
\cite[Lemma 3.9.34]{Jacob}.) The claim follows now from Theorem
\ref{stilthh}.
\end{proof}
\begin{remark}\label{remsup}
 Note that Corollary \ref{corfirst} can be formulated in terms of the norm estimates for $\Ce_t(A)x$  rather than
an integral condition on $\|\Ce_t(A)x\|.$ Indeed, \eqref{first} is
equivalent to the condition
\begin{equation}
\|\Ce_t(A)x\|= \rmO\left(\frac{1}{\epsilon(t) g(1/t)}\right),\qquad t\to \infty,
\label{qucond1}
\end{equation}
where $\epsilon:(1,\infty)\to (0,\infty)$ is a  measurable  function satisfying
\[
\int_1^\infty \frac{d\tau}{\tau\epsilon(\tau)} <\infty.
\]
\end{remark}

Observe that if $g(z)=z^{-\alpha}, \alpha \in (0,1),$ then
\eqref{qucond} and \eqref{qucond1} are, in a sense, equivalent.
Indeed, if \eqref{qucond} holds then setting $\tilde
{\epsilon}(\tau)=\epsilon(\tau^\alpha)$ we have
$$
\int_{1}^{\infty}\frac {d\tau}{\tau \tilde\epsilon(\tau)}
=\frac{1}{\alpha}\int_{1}^{\infty}\frac{d\tau}{\tau \epsilon(\tau)},
$$
and \eqref{qucond1} is satisfied with $\epsilon$ replaced
by $\tilde{\epsilon}.$ Conversely,  if $\eqref{qucond1}$ holds
then setting $\tilde {\epsilon}(\tau)=\epsilon(\tau^{1/\alpha})$
we infer that \eqref{qucond1} is true with $\epsilon$ replaced by
$\tilde{\epsilon}.$

Using Remark \ref{remsup} we  state now the following straightforward  consequence of Corollary \ref{corfirst}.
\begin{corollary}\label{stcor2}
If for $x \in X$ there exists $\alpha\in (0,1)$ such that
\begin{equation}\label{glog}
\|C_t(A)x\| =\rmO\left(\frac{1}{{g(1/t)\log^{1+\alpha}(2+g(1/t)})}\right),\qquad t\to\infty,
\end{equation}
then $x\in \dom (g(A))$.
\end{corollary}

Specifying \eqref{glog} for a power function we obtain the domain/range
condition for fractional powers of $A$.
\begin{corollary}\label{stcor3}
If for $x \in X$ there exist $\alpha,\beta\in (0,1)$ such that
\[
\|C_t(A)x\|=\rmO\left(\frac{1}{t^\beta\log^{1+\alpha}t}\right),\qquad  t\to\infty,
\]
then  $x\in \dom (A^{-\beta})=\ran (A^\beta)$.
\end{corollary}

The next result proposes a different ideology for proving the inverse mean ergodic theorems with  rates.  To be able to place
an element $x$ into  $\dom (g(A))$ in the results above we had to add
an ``extra rate'' to the rate $r(t)=(g(1/t))^{-1}$ of the decay of $\Ce_t(A)x$
obtained in Theorem \ref{ratebern}.
 Now, instead of adding an ``extra rate'',
we add an ``extra domain'' to $\dom (g(A))$. In this way, we
will show that $x$ belongs to a slightly larger space than $\dom(g(A))$ under the assumption $\norm{\Ce_t(A)x}=\rmO((g(1/t))^{-1})$. (Recall that
by Theorem \ref{ratebern} $\norm{\Ce_t(A)x}=\rmO((g(1/t))^{-1})$ would follow from merely $x \in \dom (g(A)).$)
To this aim, recall first (from Section $2$) that for any complete
Bernstein function $f$ and Stieltjes function $g$ the function $f
\circ g$ is a Stieltjes. Moreover, by \eqref{inclusion}, we have
$$\dom ((f\circ g)(A)) \supset \dom (g(A)),$$
and the inclusion is in general strict. While the
assumption $\norm{\Ce_t(A)x}=\rmO((g(1/t))^{-1})$ does not imply
$x \in \dom (g(A))$ we prove that it does suffices to guarantee $x \in  \dom ((q\circ
g)(A))$ for a large class of Bernstein functions $q.$

\begin{theorem}\label{stcor4}
Suppose $q$ is a complete Bernstein function such that
\begin{equation}
\int_1^\infty\frac{q(\tau)}{\tau^2}\,d\tau<\infty, \qquad
\lim_{s\to 0+}\,q(s)=0. \label{stilcor8}
\end{equation}
If $x \in X$  satisfies
\begin{equation}
\label{stil08}
\|\Ce_t(A)x\|=\rmO\left(\frac{1}{g(1/t)}\right),\qquad t\to\infty,
\end{equation}
then $ x\in \dom ((q \circ g)(A)). $ In particular, if
\eqref{stil08} holds, then for any $\alpha\in (0,1)$ one has $
x\in \dom ([g^\alpha](A)). $
\end{theorem}

\begin{proof}
Since $\lim_{s \to \infty} g(s)=0,$ from our assumption on $q$ it
follows that
\begin{equation}\label{li}
\lim_{s\to \infty}\, q(g(s))=0.
\end{equation}
Furthermore $q(t)/t$ is a Stieltjes function, hence it
decreases on  $(0,+\infty)$ and $\lim_{t \to \infty} q(t)/t$
exists and finite. By \eqref{stilcor8},
we have
$
\lim_{t \to \infty}{q(t)}/{t}=0.
$
Moreover,
\begin{equation*}
g(1/t)=t\int_{0+}^\infty \frac{\mu(ds)}{1+ts}
\le t\int_{0+}^\infty \frac{\mu(ds)}{1+s},\qquad t\ge 1.
\end{equation*}
Thus since $q$ is increasing we obtain
\begin{equation}\label{inn}
q(g(1/t))\le q(d(\mu)t), \qquad t\ge 1, \quad d(\mu):=
\int_{0+}^\infty \frac{\mu(ds)}{1+s}.
\end{equation}
Therefore,
\begin{equation}\label{li1}
\lim_{s\to 0+}\,s q(g(s))\leq \lim_{t\to \infty
}\frac{q(d(\mu)t)}{t}=0.
\end{equation}
Using \eqref{li} and \eqref{li1} we infer that  the Stieltjes
function $q \circ g$
 has the representation $(0,0,\nu).$

Next we apply Theorem \ref{stilprop1} to $x$ and $(q \circ g)(A).$
Using the hypothesis on the decay of $\|\Ce_t(A)x\|$ and  the estimate
\eqref{TC0} from the proof of Theorem \ref{stilthh}, we obtain
\begin{eqnarray*}
\int_{0+}^1 \|(A+s)^{-1}x\|\,\nu(ds)
&\leq& M\|x\|\int_{0+}^1 \nu(ds)\\
&+& 4c\int_1^\infty\frac{1}{g(1/t)} (q(g(1/t))' \, dt,
\end{eqnarray*}
for some $c >0.$
Furthermore
\begin{eqnarray*}
\int_1^\infty\frac{1}{q(1/t)}(q(g(1/t))'\,dt&=&
\int_1^\infty\frac{q{'}(g(1/t))}{g(1/t)} d g(1/t)
=\int_{\tau_0}^\infty\frac{q{'}(\tau)}{\tau}\,d\tau\\
&=&-\frac{q(g(1))}{g(1)}+
\int_{\tau_0}^\infty \frac{q(\tau)}{\tau^2}\,d\tau
\leq \int_{g(1)}^\infty
\frac{q(\tau)}{\tau^2}\,d\tau.
\end{eqnarray*}
Thus finally
\[
\int_{0+}^1 \|(A+s)^{-1}x\|\,\nu(ds)\leq M\|x\|\int_{0+}^1
\nu(ds)+4c\int_{g(1)}^\infty \frac{q(\tau)}{\tau^2}\,d\tau
<\infty,
\]
and then $ x\in \dom \,((q \circ g)(A)). $ The last statement
follows from the fact that $z^{\alpha}, \alpha \in (0,1),$ is a
complete Bernstein function satisfying \eqref{stilcor8}.
\end{proof}

 Note that if $g \sim (0,0,\mu)$ is a Stieltjes  function
 then
\begin{equation}
g(t)=\int_0^\infty e^{-t s} m(s)\,ds,\quad
m(s)=\int_0^\infty e^{-s \tau}\,\mu(d\tau),\quad t,s>0,
\label{Ap11}
\end{equation}
where $m$ is integrable in the neighborhood of zero and $\lim_{t \to \infty} m(t)=0.$

Using \eqref{Ap11} and our direct mean ergodic theorems with rates, we prove below a
characterization of $\dom \, (g(A))$ in terms of $\Ce_t(A)$ which complements
\cite[Corollaire, p. 214-215]{HirschFA}. It involves certain means of $\Ce_t(A)$ thus avoiding a need of adding ``extra rate'' or ``extra range''. It is however less explicit than theorems above.
\begin{proposition}\label{chardommean}
An element
 $x$ belongs to $\dom(g(A))$ if and only if
\begin{itemize}
\item [{\it (i)}] $\displaystyle{ \lim_{t\to \infty}t m(t) \Ce_t(A)x=0};$
\item [{\it (ii)}] $\displaystyle{\lim_{t\to\infty}\,\int_0^t s m{'}(s) \Ce_s(A)x\,ds}\,\,$
exists.
\end{itemize}
\end{proposition}
\begin{proof}
By  \cite[Corollaire, p. 214-215]{HirschFA}
 $x\in \dom(g(A))$ if and only if
\begin{equation}
\lim_{t\to\infty}\,\int_0^t m(s) T(s)x\,ds \quad  \text{exists}.
\label{ApL10}
\end{equation}
Since for any $x\in X$ and $t>0$
\begin{equation}
\int_0^t m(s) T(s)x\,ds=t m(t) \Ce_t(A)x- \int_0^t s
m{'}(s)\Ce_s(A)x\,ds, \label{ApN0}
\end{equation}
it suffices to show that \eqref{ApL10} implies {\it (i)}
and {\it (ii).}
Using \eqref{Ap11}
and $e^{-s}\le 1/(1+s)$, $s\ge 0,$
we infer that
\begin{equation}
s m(s)\le g(1/s),\quad s>0.
\label{Apin}
\end{equation}
Since by assumption $\ran(A)$ is dense,
\cite[Theorem 3.4 and Proposition 4.2]{GHT1} imply that
$g(1/t)\Ce_t(A)x \to 0, t \to \infty,$ and then
 \eqref{Apin} implies {\it (i)}.
If {\it (i)} holds, then from \eqref{ApN0} it follows that {\it (ii)} holds as well.
\end{proof}

If  $g(z)=z^{-\alpha}$, $\alpha\in (0,1)$,
then we can derive a slightly stronger result which is a continuous counterpart of
\cite[Lemma 4.1]{Cu10}.

\begin{theorem}\label{thmcrit}
If $\alpha\in (0,1)$ then  $x\in \dom(A^{-\alpha})$ if and only if
\begin{equation}
\lim_{t\to\infty}\,\int_0^t  s^{\alpha-1} \Ce_s(A)x\,ds \quad
\text{exists}. \label{ApL2}
\end{equation}
\end{theorem}

\begin{proof}
Since
\[ z^{-\alpha} = \frac{1}{\Gamma(\alpha)} \int_0^\infty s^{\alpha -1}
e^{-sz}\, ds, \qquad z >0,
\]
we have $m(s)= \frac{1}{\Gamma(\alpha)} s^{\alpha-1}$, $s>0$,
and by Proposition \ref{chardommean}  it suffices to prove
that \eqref{ApL2} implies
\begin{equation}
t^{\alpha} \Ce_t(A)x\to 0,\qquad t\to\infty.
\label{ApNN}
\end{equation}

To this aim note that if \eqref{ApL2} holds,
then
\begin{eqnarray*}
\int_t^{2t}  s^{\alpha-1} \Ce_s(A)x\, ds
&=&\int_t^{2t} s^{\alpha-2}ds\int_0^t T(\tau)x\,d\tau\\
&+&T(t)\int_0^{t} (r+t)^{\alpha-2}\int_0^r T(\tau)x\,d\tau dr\\
&=& \frac{(1-2^{\alpha-1})}{1-\alpha}\,t^\alpha \Ce_t(A)x
+T(t)\int_0^{t} (r+t)^{\alpha-2}r \Ce_r(A)x dr,
\end{eqnarray*}
where the last sum goes to zero as $t \to \infty.$
Thus to prove \eqref{ApNN} it suffices to show that
\begin{equation*}
\lim_{t\to\infty}\,\int_0^{t} (r+t)^{\alpha-2} r\Ce_r(A)x\, dr=0.
\label{Apst}
\end{equation*}
Setting
\begin{eqnarray*}
G_t(r):=\frac{r^{2-\alpha}}{(r+t)^{2-\alpha}}, \,\, (t>0 \,\,
\text{is fixed}); \qquad R(r):=\int_r^\infty s^{\alpha-1}
\Ce_s(A)x\,ds, \qquad r \ge 0,
\end{eqnarray*}
where the second function is well-defined, continuously differentiable  and bounded on $[0,\infty)$ by our assumption,
write
\begin{equation}\label{summ}
\int_0^{t} (r+t)^{\alpha-2}r \Ce_r(A) x \, dr =- G_t(t)  R(t)+ \int_0^t
G^{'}_t(r) R(r) \, dr.
\end{equation}
We prove that both terms on the right hand side of \eqref{summ} converge to zero as $t  \to \infty,$ and
thus obtain the statement.

First note that for all $t >0$ and $r >0$
\begin{equation}
0<G_t(r)\le 1,\qquad  r>0; \quad \lim_{t\to\infty}\,G_t(r)=0.
\label{ApGN}
\end{equation}
Hence, by our assumption,
$
\lim_{t \to \infty}\norm{G_t(t)R(t)}=0.
$
To prove the convergence to zero of the other term note that
\[
G^{'}_t(r)=(2-\alpha)\left(\frac{r}{r+t}\right)^{1-\alpha}\frac{t}{(r+t)^2}
\]
is positive on $(0,\infty)$ for each $t >0.$
Let $\epsilon >0$ be fixed. Then by assumption there exists $b=b(\epsilon)$ such that
$
\|R(t)\|\le \epsilon
$
if $t \ge b.$
Now using \eqref{ApGN}
we have for large enough $t$
\begin{eqnarray*}
\left\|\int_0^t G^{'}_t(r) R(r) \, dr\right\|&\le& \int_0^b
G^{'}_t(r) \|R(r)\|\, dr+
\int_b^T G^{'}_t(r) \, dr\,\sup_{r\ge b}\|R(r)\|\\
&\le& \sup_{r\ge b}\|R(r)\| G_T(b) +
\epsilon G_t(t)<2\epsilon,
\end{eqnarray*}
and the statement follows.
\end{proof}

\section{Optimality of domain conditions}\label{optimality}

In this section we give a number of results showing that the
inverse theorems on rates proved in the previous sections are
optimal. The results will illustrate, in particular, that the
implications of the form
\[
\norm{C_t(A)x}=\mbox{O}(1/g(1/t)),\quad t\to\infty\quad
\Rightarrow\quad x\in \dom(g(A)),
\]
are far from being true, in general, and direct theorems on rates
obtained in \cite{GHT}  cannot be inverted.
 Thus to get positive statements one has to add
either ``extra rate'' or ``extra range'' assumptions as it was done above.

We start with introducing  basic objects  for constructing our
examples. Let $L_1:=L_1(1,\infty)$, and define a bounded
operator $A$ on $L_1$ by
\begin{equation}
(Au)(s):=\frac{u(s)}{s}, \qquad u\in L_1. \label{operator}
\end{equation}
Note that the $-A$ generates a contraction $C_0$-semigroup $(T(t))_{t \ge 0}$ given by
\begin{equation}
(T(t)u)(s):=e^{-t/s}u(s),\qquad t \ge 0.
\label{stilform1}
\end{equation}
Thus we have in particular
\[
(\Ce_t(A)u)(s)=\frac{1}{t}\int_0^t e^{-\tau/s}u(s)d\tau
=\frac{s(1-e^{-t/s})}{t}u(s),
\]
and
\begin{equation}
t\|\Ce_t(A)u\|_{L_1}=\norm{w_t u}_{L_1},\quad w_t(s):=s(1-e^{-t/s}) \qquad u\in L_1.
\label{stilform}
\end{equation}
A direct application of functional calculi rules reveals that
 for any Stieltjes function
$g$ the operator $g(A)$ is
of the form
\begin{eqnarray*}
 \label{domain} \dom(g(A))&=&\left\{u\in L_1:
\int_1^\infty  g(1/s)|u(s)|\,ds<\infty\right\},\\
(g(A)u)(s)&=&g(1/s)u(s), \qquad \text{for a.e.} \,\, s \ge 1.\notag
\end{eqnarray*}

It will be convenient to introduce the following family of norms on $L_1$:
\begin{equation}
N_{t}(u):=\frac{1}{t}\int_1^t s |u(s)|\,ds+
\int_t^\infty |u(s)|\,ds,\quad \qquad u\in L_1, \qquad t\geq 1.
\label{norms}
\end{equation}
The norms are equivalent to the original norm on $L_1:$
\begin{equation}
t^{-1}\|u\|_{L_1}\le N_{t}(u)\le \|u\|_{L_1},\qquad u\in L_1.
\label{norms1}
\end{equation}
Moreover, using the  inequalities
\begin{eqnarray*}
(1-e^{-1})\tau &\le& 1-e^{-\tau}\leq \tau,\qquad \tau\in (0,1),\\
1-e^{-1}&\le& 1-e^{-\tau}\leq 1,\qquad \tau\ge 1,
\end{eqnarray*}
and \eqref{stilform}, we obtain
\begin{equation}
(1-e^{-1})N_t(u)\leq \|C_t(A)u\|_{L_1}\leq N_{t}(u),\qquad t>1,\quad u\in L_1,
\label{es1}
\end{equation}

Let in this section $$ g\,\, \text{\emph{be a Stieltjes function
such that}}\,\, g
\sim (0,0,\mu) \,\, \text{\emph{and}}\,\, g(0+)=\infty.$$

First we show that Theorem \ref{stilthh} is sharp in the sense that for $A$ defined by \eqref{operator} and for a large class of $g$ the properties
\eqref{g1} and $x \in \dom (g(A))$ are in fact equivalent.
\begin{theorem}\label{L1}
The condition
\begin{equation}
\int_1^\infty\frac{|g{'}(1/t)| \norm{C_t(A)u}_{L_1}}{t^2}dt<\infty
\label{equiv1}
\end{equation}
holds  for all $u \in \dom(g(A))$ if and only if
$g$ is integrable on $(0,1)$ and there exists $c >0$ such that
\begin{equation}
\frac{1}{t}\int_0^t g(\tau)\,d\tau \leq c g(t), \qquad t\in (0,1).
\label{mainc6}
\end{equation}
\end{theorem}

\begin{proof}
Assume that \eqref{equiv1} holds for all $u\in \dom(g(A)).$ Since
the operator $g(A)$ is closed, $\left(\dom (g(A)),
\norm{\cdot}_{\dom(g(A))}\right)$ is a Banach space with the graph
norm
\[
\norm{u}_{\dom(g(A))}:=
\norm{g(A)u}_{L_1}+\norm{u}_{L_1}=\int_1^\infty (g(1/s)+1)
|u(s)|\,ds.
\]
Consider a linear operator
\begin{eqnarray*}
G &:&\dom(g(A))\mapsto L_1 \left((1,\infty); L_1\right),\\
G u &:=& \frac{|g{'}(1/t)| C_t(A)u}{t^2}, \qquad u \in \dom
(g(A)).
\end{eqnarray*}
From our assumption it follows that $G$ is well defined. By a
standard argument (after passing to
appropriate a.e. convergent subsequences) $G$ is
closed. Therefore, by the closed graph theorem and \eqref{es1}, it follows that
\begin{equation}
\int_1^\infty\frac{|g{'}(1/t)|  N_t(u)}{t^2}\,dt\leq
c\norm{u}_{\dom(g(A))},\label{aest}
\end{equation}
for some constant $c>0$ not depending on  $u\in \dom(g(A))$.

Using \eqref{norms} and Fubini's theorem we can write down the right hand side of \eqref{aest} as follows:
\begin{eqnarray*}
\int_1^\infty\frac{|g{'}(1/t)|  N_t(u)}{t^2}\,dt
&=&\int_1^\infty\frac{|g{'}(1/t)|}{t^2}
\left\{\frac{1}{t}\int_1^t s |u(s)|\,ds+\int_t^\infty |u(s)|\,ds\right\} dt\\
&=&\int_1^\infty |u(s)| W_g(s)\,ds,
\label{best}
\end{eqnarray*}
where
\begin{eqnarray*}
W_g(s)&:=&-s\int_s^\infty \frac{g{'}(1/t)}{t^3} dt-\int_1^s \frac{g{'}(1/t)}{t^2} dt.
\label{Wu}
\end{eqnarray*}
Integrating by parts and  taking into account that
$\lim_{\tau\to 0+}\,\tau g(\tau)=0$ (by the bounded convergence
theorem) we infer that $g \in L_1(0,1)$ and for  every $s \ge 1$
\begin{equation}\label{www}
W_s(g)=-s\int_0^{1/s} g{'}(\tau)\tau d\tau-
\int_{1/s}^1 g{'}(\tau) d\tau
=s\int_0^{1/s} g(\tau)\,d\tau-g(1).
\end{equation}
Thus \eqref{aest} is satisfied for all $u \in \dom (g(A))$
if and only if
\begin{equation}
\int_1^\infty W_g(s)|u(s)|\,ds\leq c \int_{1}^{\infty}(g(1/s)+1) |u(s)|\, ds, \qquad u\in \dom(g(A)),
\label{aest1}
\end{equation}
where $W_s(g)$ is given by \eqref{www}.
In turn, \eqref{aest}  is equivalent to
\begin{equation}
W_g(s)\le c(g(1/s)+1),\qquad  s \ge 1. \label{W_sc}
\end{equation}
Indeed, it suffices to note that $\dom(g(A))$ contains all
integrable functions with compact support.
Writing down \eqref{aest1} for $u=1_{(a,b)}$, $1\le a<b<\infty$,
we obtain that
\[
\int_a^b W_g(s)\,ds\leq c \int_a^b (g(1/s)+1) \, ds.
\]
Hence
\[
F(s):=\int_1^s [c(g(1/t)+1)-W_g(t)]\,dt,\quad s\ge 1,
\]
is an increasing function, and then $F'(s)=c(g(1/s)+1)-W_g(s)\ge 0$, $s\ge 1$.

Therefore, \eqref{W_sc} can be rewritten as
\[
s\int_0^{1/s} g(\tau)\,d\tau-g(1)\le c(g(1/s)+1),\qquad s\ge 1,
\]
that is
\[
\frac{1}{t}\int_0^t g(\tau)\,d\tau\le  cg(t)+c+g(1),\qquad t\in (0,1).
\]
Since the function $g$  is decreasing on $(0,\infty),$ the last
inequality is equivalent to
 \eqref{mainc6} (with in general a new constant $c>0$).
\end{proof}

\begin{remark}
A  natural question is what are the functions $g$ satisfying \eqref{mainc6}.
To show that the class of such functions is quite large
we note that if
\begin{equation}
\tau^\alpha g(\tau)\quad \mbox{is increasing on}\,\,\,
(0,1) \label{lemmac1}
\end{equation}
for some $\alpha \in (0,1)$, then \eqref{mainc6} is satisfied, e.g
$g$ could be a power function. Indeed, if
\eqref{lemmac1} holds, then
\begin{equation}\label{alpha}
\int_0^t g(\tau)\,d\tau\le
 t^\alpha g(t)\int_0^t \frac{d\tau}{\tau^\alpha}=
\frac{t}{1-\alpha}g(t),\qquad t\in (0,1).
\end{equation}
On the other hand there are Stieltjes functions $g \sim (0,0,\mu)$
with $g(0+)=\infty$ for which \eqref{mainc6} is not true.
For instance, if
$
g(z):=\frac{z-1}{z\log z}
$
then  $g \not \in L_1(0,1).$
\end{remark}

Now we prove that Corollary \ref{corfirst} is optimal.

\begin{theorem}\label{L2}
The condition
\begin{equation}
\int_1^\infty\frac{g(1/t)\norm{C_t(A)u}_{L_1}}{t}dt<\infty
\label{equiv2a}
\end{equation}
holds for all $u\in \dom(g(A))$  if and only if $g$ is integrable
on $(0,1)$ and there exists $c >0$ such that
\begin{equation}
\frac{1}{t}\int_0^t g(\tau)\,d\tau +
\int_t^1\frac{g(\tau)}{\tau}\,d\tau \leq c g(t),\quad t\in
(0,1). \label{seccond}
\end{equation}
\end{theorem}

\begin{proof}
The proof is similar to the proof of  Theorem \ref{L1} and will
only be sketched. Note that
\begin{eqnarray*}
\int_1^\infty\frac{g(1/t)  N_t(u)}{t}\,dt
&=&\int_1^\infty\frac{g(1/t)|}{t}
\left\{\frac{1}{t}\int_1^t s |u(s)|\,ds+\int_t^\infty |u(s)|\,ds\right\} dt
\\
&=&\int_1^\infty |u(s)| V_s(g)\,ds,
\\
\text{where}\quad  \qquad  V_s(g)&:=& s\int_0^{1/s} g(\tau)\,d\tau
+\int_{1/s}^1 \frac{g(\tau)}{\tau}\, d\tau.
\end{eqnarray*}
Then by the argument analogous to that from the proof of Theorem
\ref{L1},  \eqref{equiv2a} is equivalent to the
inequality
\begin{equation}
V_s(g)\le c(g(1/s)+1),\qquad s\ge 1,
\label{V_sc}
\end{equation}
for some constant $c >0,$  which in turn is equivalent to
\eqref{seccond}.
\end{proof}

\begin{remark}
Observe  that if there exist $\alpha,\beta \in (0,1)$ such that
$0< \beta < \alpha < 1$ and
\begin{equation}
\tau^\alpha g(\tau)\quad \mbox{is increasing on} \,\, (0,1), \,\,
\tau^\beta g(\tau)\quad \mbox{is decreasing  on}\,\, (0,1),
\label{lemmac10}
\end{equation}
 then
\eqref{seccond} holds. Indeed, in view of \eqref{alpha} it suffices to note that
\[
\int_t^1 \frac{g(\tau)}{\tau}\, d\tau\le
 t^\beta g(t)\int_t^1 \frac{d\tau}{\tau^{1+\beta}}\le
\frac{g(t)}{\beta},\qquad t\in (0,1).
\]
\end{remark}

\begin{example}
Observe that the Stieltjes function $ g(z)=(z-1)^{-1}\log z $
satisfies \eqref{mainc6} but it does not satisfy \eqref{seccond}.
Indeed, if $t\in (0,1/2)$ then we have
\begin{eqnarray*}
\frac{1}{t}\int_0^t \frac{\log \tau\,d\tau}{\tau-1}\le
-\frac{2}{t}\int_0^t \log \tau\,d\tau
\le 2\frac{\log t}{t-1}=2g(t).
\end{eqnarray*}

On the other hand,
\[
\int_t^1 \frac{\log \tau\,d\tau}{(\tau-1)\tau}\ge
-\int_t^1 \frac{\log \tau\,d\tau}{\tau}=\frac{\log^2 t}{2},
\]
and \eqref{seccond} is violated.
\end{example}

Thus if $g$ satisfies the conditions of Theorem \ref{L2} and $u
\not \in \dom (g(A))$ then the integral in \eqref{equiv2a} diverges. In this case \eqref{equiv2a} can hardly be written in
the form of sup-norm estimates. To circumvent this drawback, we
use the notion of slowly varying function. Recall  that (see \cite[Chapter 1]{Seneta}) a measurable
function $\epsilon: (a,\infty)\mapsto (0,\infty)$, $a\ge 0$, is
called {\it slowly varying} (at infinity) if for all $\lambda>0$ one has
$
\lim_{t\to+\infty}\,{\epsilon(\lambda t)}/{\epsilon(t)}=1.
$
 We proceed with a result which in a sense complements
Theorem \ref{L2}.

\begin{theorem}\label{L10}  Assume that
\eqref{lemmac1} holds. Let $\epsilon$ be a slowly varying on
$(g(1),\infty)$ function. Then the function
\[
y(s):=-\frac{g{'}(1/s)}{s^2g^2(1/s)\epsilon(g(1/s))},\qquad s>1,
\]
is positive,  belongs to $L_1,$ and satisfies
\begin{equation}
N_{t}(y)={\rm O}\left(\frac{1}{g(1/t)\epsilon(g(1/t))}\right),\qquad t\to\infty.
\label{estim1}
\end{equation}
Moreover,  $y\in \dom(g(A))$ if and only if
\begin{equation}
\int_{g(1)}^\infty\frac{d\tau}{\tau \epsilon(\tau)}<\infty.
\label{epstheorem}
\end{equation}
\end{theorem}

\begin{proof}
Since $g(0+)=\infty,$ by
the change of  variable $\tau=g(1/s), \tau \in (g(1/t),\infty)$, we obtain
for any $t\ge 1$:
\begin{eqnarray}
\int_t^\infty y(s)\,ds&=&-\int_t^\infty
\frac{g{'}(1/s)\,ds}{s^2 g^2(1/s)\epsilon(g(1/s))}
=
\int_{g(1/t)}^\infty \frac{d\tau}{\tau^2 \epsilon(\tau)}
\label{*}\\
&\le& \frac{1}{\epsilon(g(1/t)}\int_{g(1/t)}^\infty \frac{d\tau}{\tau^2}
= \frac{1}{ g(1/t) \epsilon(g(1/t))},\notag
\label{est2}
\end{eqnarray}
hence $y\in L_1$.
Similarly,
\[
\int_1^\infty g(1/s)y(s)\, ds=
-\int_1^\infty \frac{g{'}(1/s)\,ds}{s^2g(1/s)\epsilon(g(1/s))}=
\int_{g(1)}^\infty \frac{d\tau}{\tau \epsilon(\tau)},
\]
and, in view of the description of $\dom (g(A))$ in \eqref{domain}, the statement  follows.

It remains to prove that $y$ satisfies \eqref{estim1}.
Observe that
by means of
\eqref{norms} and \eqref{*} one can rewrite \eqref{estim1} as
\begin{equation}
\frac{1}{t}\int_1^t s y(s)\,ds\leq
 \frac{c}{g(1/t)q(g(1/t))},\quad t\ge 1,
\label{estnew}
\end{equation}
for some constant $c > 0.$ By \cite[Section 1.5]{Seneta},
 for any $\delta>0$
the function $t^{-\delta}\epsilon(t)$ is equivalent as $t \to \infty$ to a positive function decreasing on $(g(1),\infty)$,
and the function $t^{\delta}\epsilon(t)$ is equivalent as $t\to \infty$ to a positive function increasing on $(g(1),\infty).$
Therefore, since $g(0+)=\infty$ and $g$ is decreasing, for any $\delta >0,$
$
g^{-\delta}(\tau)\epsilon (g(\tau))$ is equivalent to a function increasing  on $(1,\infty).
$
Choose now positive $\delta$  such that
$\beta:=(1+\delta)\alpha\in (0,1),$ where $\alpha$ is defined in \eqref{lemmac1}. Then
\[
\tau^\beta g(\tau)\epsilon(g(\tau))= (\tau^\alpha
g(\tau))^{1+\delta} g^{-\delta}(\tau)\epsilon(g(\tau))
\]
is equivalent to a function increasing on $(0,1).$ In other words,
the function $g(1/s)\epsilon(g(1/s))$ is equivalent to a
measurable function $\psi$ such that
$
{s^\beta }/{\psi (s)}$ is increasing on  $(1,\infty).
$
Hence since $\tau |g'(\tau)|\le g(\tau), \tau >0,$  by
\eqref{rem3},  we obtain for every $t\ge 1$:
\begin{eqnarray*}
\frac{1}{t}\int_1^t s y(s)\,ds&=&\frac{1}{t}\int_1^t \frac{|g{'}(1/s)|\,ds}{sg^2(1/s)\epsilon(g(1/s))}
\le \frac{1}{t}\int_1^t \frac{ds}{g(1/s)\epsilon(g(1/s))}\\
&\le& \frac{c}{t}\int_{1}^{t}\frac{s^\beta \, ds}{s^\beta \psi (s)}
\le
\frac{c}{t} \frac{t^\beta}{\psi (t)} \int_{1}^{t}\frac{ds}{s^\beta}\\
&\le&
C \frac{t^\beta}{t g(1/t) \epsilon(g(1/t))}\int_1^t \frac{ds}{s^\beta}
\le C \frac{1}{(1-\beta) g(1/t) \epsilon(g(1/t))},
\end{eqnarray*}
where $c, C$ are positive constants,
thus the proof is complete.
\end{proof}

Theorem \ref{L10} and \eqref{es1}
imply the following statement (cf. Corollary
\ref{qgt}).

\begin{corollary}\label{C_t}
Assume that the functions $g$ and  $\epsilon$ satisfy the conditions of
Theorem \ref{L10}. If
\[
\int_{g(1)}^\infty \frac{d\tau}{\tau\epsilon(\tau)}=\infty,
\]
then  there exists  $y\in L_1$ such that
\begin{equation}\label{supest}
\|\Ce_t(A)y\|_{L_1}=\rmO\left(\frac{1}{g(1/t)\epsilon(g(1/t))}\right),\quad
t\to\infty,
\end{equation}
but $y\not\in \dom(g(A))$.
\end{corollary}

\begin{remark}\label{rempdop}
By \cite[Theorem 4.6]{GHT} if  $\epsilon: (0,\infty)\to (0,\infty)$ is
an increasing   function such that $\lim_{t\to
\infty}\,\epsilon(t)=\infty$, then there exists $x\in \dom(g(A))$ such
that
\[
\sup_{t\ge 1}\,
g(1/t)\epsilon(g(1/t))\norm{\Ce_t(A)x}_{L_1}=\infty.
\]
Thus conditions like \eqref{supest} cannot hold for all elements from the corresponding domain.
\end{remark}

\begin{example}\label{ex1}
Observe that $g(z)=z^{-\gamma}, \gamma \in (0,1),$ is a Stieltjes
function satisfying  \eqref{lemmac1} for $\alpha \in (\gamma, 1).$
Therefore, $g$ satisfies also \eqref{mainc6}. (The latter fact can
also be checked directly.)
 Since
\begin{equation}
\epsilon(s):=\log (s+2)\log(\log(s+3)),\qquad s\ge 0,
\label{qqq}
\end{equation}
is slowly varying on $(0,\infty),$  the functions $g$ and $\epsilon$
satisfy the conditions of Theorem \ref{L1}. Hence by Corollary
\ref{C_t} there exists
\[
y\in L_1,\qquad y\not\in\dom(A^{-\gamma}),
\]
such that
\[
\|\Ce_t(A)y\|_{L_1}=\mbox{O}
\left(\frac{1}{t^{\gamma}\log(t) \log(\log t)}\right),\;\;t\to\infty.
\]
\end{example}
\begin{example}\label{ex2}
Note that the Stieltjes function $ g(z)=\log(1+z^{-1})$
and the function $\epsilon$ defined by (\ref{qqq}) satisfy the
conditions of Theorem \ref{L10} (since $g$ satisfies
\eqref{lemmac1} for any $\alpha\in (0,1)$). Hence by Corollary
\ref{C_t}, there exists
\[
y\in L_1, \quad y \not\in\dom(\log(I+A^{-1})),
\]
such that
\[
\|\Ce_t(A)y\|_{L_1}=\mbox{O}\left(\frac{1}{\log t[\log(\log t)
\log(\log(\log t))]}\right),\quad t\to\infty.
\]
\end{example}

\section{Appendix}
Recall that if $(T(t))_{t \ge 0}$ is a bounded $C_0$-semigroup on $X$ then for each $x \in X\setminus\{0\}$ the Ces\'aro means $\Ce_t(A)x$ cannot decay faster than $1/t$
as $t \to \infty.$ The proposition below shows that it is not possible to `improve' this extremal rate of decay  of $\Ce_t(A)x$ by requiring the smallness of $\Ce_t(A)x$
in an integral sense.

\begin{proposition}\label{cesarozero1T}
Let  $(T(t))_{t\ge 0}$ be a bounded $C_0$-semigroup
on a Banach space $X$ with generator $-A$.
If for $x \in X$ there exists $\{t_k: k \ge1 \}\subset (0,\infty),$  $t_k\to\infty$, $k\to\infty$,
such that
\begin{equation}
\mbox{weak}-\lim_{k\to \infty} \,\frac{1}{t_k}\int_0^{t_k} s \Ce_s(A)x\,ds\,=0,
\label{zero12T}
\end{equation}
then $x=0$.
In particular, if
\begin{equation*}
\mbox{weak}-\lim_{t\to\infty}\, t\Ce_t(A)x=0,
\end{equation*}
then $x=0.$
\end{proposition}

\begin{proof}
Since
\[
 tA(I+A)^{-1} \Ce_t(A)x=(I-T(t))(I+A)^{-1}x,\;\;t>0,
\]
we have
\[
[A(I+A)^{-1}]^2\frac{1}{t}\int_0^t s \Ce_s(A)x\,ds=
A(I+A)^{-2}x-\frac{(I-T(t))}{t}(I+A)^{-2}x.
\]
As the operator $A(I+A)^{-1}$ is bounded,
the latter equality and  (\ref{zero12T}) imply that
$A(I+A)^{-2}x=0$ and then $x\in \ker{A}$.
But if  $x\in \ker{A}$ then
\[
\frac{1}{t}\int_0^t s \Ce_s(A)x\,ds = \frac{1}{t}\int_{0}^{t} s \,
ds\, x = \frac{t}{2}x,
\]
and, using (\ref{zero12T}) once again, we conclude that $x=0$.
\end{proof}

\begin{theorem}\label{cesarozero1}
Let  $(T(t))_{t\ge 0}$ be a bounded $C_0$-semigroup
on a Banach space $X$ with generator $-A$.
Let $\varphi$
be a positive, increasing on $[1,\infty)$ function such that
\begin{equation}
\int_1^\infty \frac{dt}{t\varphi(t)}=\infty.
\label{phiAp}
\end{equation}
If $x \in X$ satisfies
\begin{equation}
\int_1^\infty \frac{\|\Ce_t(A)x\|}{\varphi(t)}\,dt<\infty,
\label{zero12}
\end{equation}
then $x=0$.
In particular, if
\begin{equation*}
\int_1^\infty \frac{\|\Ce_t(A)x\|}{\log(1+t)}\,dt<\infty,
\end{equation*}
then $x=0.$
\end{theorem}

\begin{proof}
Define
\[
\Theta(t):=\int_0^t s\|C_s(A)x\|\,ds, \qquad t \ge 1.
\]
If $s>1$, then
\begin{eqnarray*}
\int_1^s \frac{\|\Ce_t(A)x\|}{\varphi(t)}dt&=&
\int_1^s \frac{1}{t\varphi(t)}d\Theta(t)
\\
&=&\frac{\Theta(s)}{s\varphi(s)}-\frac{\Theta(1)}{\varphi(1)}-
\int_1^s \Theta(t) d\left(\frac{1}{t\varphi(t)}\right)\\
&=&\frac{\Theta(s)}{s\varphi(s)}-\frac{\Theta(1)}{\varphi(1)}+
\int_1^s \frac{\Theta(t)}{t^2\varphi(t)} dt-
\int_1^s \frac{\Theta(t)}{t} d\left(\frac{1}{\varphi(t)}\right)\\
&\ge& -\frac{\Theta(1)}{\varphi(1)}+
\int_1^s \frac{\Theta(t)}{t^2\varphi(t)} dt.
\end{eqnarray*}

Hence by (\ref{zero12}) it follows that
\begin{equation}\label{ZZZ0}
\int_1^\infty \frac{\Theta(t)\,dt}{t^2\varphi(t)}<\infty.
\end{equation}
Therefore by  (\ref{phiAp})
and (\ref{ZZZ0}) we infer
that there exists  $t_k\to\infty$, $k\to\infty,$ such that
$
\lim_{k\to\infty}\,\Theta(t_k)/t_k=0.
$
Therefore (\ref{zero12T}) holds
and by  Proposition
\ref{cesarozero1T} we have $x=0$.
\end{proof}

\begin{remark}\label{ApB}
Note that if $g$ as in Section \ref{section3} and $g \sim (0,b,\mu), b >0,$ then $g'(1/t)t^{-2}$ and $g(1/t)/t$ are separated from zero
on $(0,\infty)$ so that the conditions \eqref{g1} and \eqref{first} reduce to \eqref{zero12}
with $\varphi(t)\equiv 1$ yielding $x=0.$
\end{remark}
\begin{center}
{\bf Acknowledgements}
\end{center}
\vspace{0,2cm}

The authors are grateful to  M. Haase and M. Lin for useful remarks and fruitful discussions. They would also like to thank
M. Lin for sending them the unpublished manuscript \cite{CoCuLi11}.
\vspace{0,2cm}

\end{document}

%% file: hatomakro.tex
\newcommand{\Wip}{\mathrm{A}_+^1}

\newcommand{\vanish}[1]{\relax}

\newcommand{\R}{\mathbb{R}}
\newcommand{\C}{\mathbb{C}}

\newcommand{\ud}{\mathrm{d}}

\newcommand{\eM}{\mathrm{M}}

\newcommand{\Lap}{\mathcal{L}}

\newcommand{\Sum}[2][\relax]{%
 \ifx#1\relax \sideset{}{_{#2}}\sum
 \else \sideset{}{^{#1}_{#2}}\sum
 \fi}

\DeclareMathOperator{\re}{Re}

\newcommand{\abs}[1]{\left| #1 \right|}
\renewcommand{\abs}[1]{\left\vert#1\right\vert}

\newcommand{\Ce}{\mathrm{C}}

\newcommand{\pfeil}{\longrightarrow}
\newcommand{\tpfeil}{\longmapsto}
\DeclareMathOperator{\dom}{dom}
\DeclareMathOperator{\ran}{ran}

\newcommand{\cls}[1]{\overline{#1}}

\DeclareMathOperator{\Lin}{\mathcal{L}}

\newcommand{\norm}[2][\relax]{%
   \ifx#1\relax \ensuremath{\left\Vert#2\right\Vert}
   \else \ensuremath{\left\Vert#2\right\Vert_{#1}}
   \fi}

\makeatletter
\newcommand{\sprod}[2]{\ensuremath{%
  \setbox0=\hbox{\ensuremath{#2}}
  \dimen@\ht0
  \advance\dimen@ by \dp0
  \left(\left.#1\rule[-\dp0]{0pt}{\dimen@}\,\right|#2\hspace{1pt}\right)}}
 \makeatother

\newcounter{aufzi}

\newcounter{aufzii}

\newcounter{aufziii}